\documentclass[12pt,reqno,english]{smfart}

\usepackage{amsmath,amssymb,smfthm,stmaryrd,epic,enumerate,dsfont}
\usepackage[english]{babel}
\usepackage[all]{xy}
\usepackage{csquotes}
\usepackage{mathabx}
\usepackage{mathrsfs}
\usepackage[active]{srcltx}

\newcommand*{\htarrow}{\lhook\joinrel\twoheadrightarrow}

\NumberTheoremsIn{subsection}\SwapTheoremNumbers

\begin{document}

\newcommand{\REMARK}[1]{\marginpar{\tiny #1}}
\newtheorem{thm}{Theorem}[subsection]
\newtheorem{lemma}[thm]{Lemma}
\newtheorem{corol}[thm]{Corollary}
\newtheorem{propo}[thm]{Proposition}
\newtheorem{defin}[thm]{Definition}
\newtheorem{Remark}[thm]{Remark}
\numberwithin{equation}{subsection}

\newtheorem{notas}[thm]{Notations}
\newtheorem{nota}[thm]{Notation}
\newtheorem{defis}[thm]{Definitions}
\newtheorem*{thm*}{Theorem}
\newtheorem*{prop*}{Proposition}
\newtheorem*{conj*}{Conjecture}

\def\Tm{{\mathbb T}}
\def\Um{{\mathbb U}}
\def\Am{{\mathbb A}}
\def\Fm{{\mathbb F}}
\def\Mm{{\mathbb M}}
\def\Nm{{\mathbb N}}
\def\Pm{{\mathbb P}}
\def\Qm{{\mathbb Q}}
\def\Zm{{\mathbb Z}}
\def\Dm{{\mathbb D}}
\def\Cm{{\mathbb C}}
\def\Rm{{\mathbb R}}
\def\Gm{{\mathbb G}}
\def\Lm{{\mathbb L}}
\def\Km{{\mathbb K}}
\def\Om{{\mathbb O}}
\def\Em{{\mathbb E}}
\def\Xm{{\mathbb X}}

\def\BC{{\mathcal B}}
\def\QC{{\mathcal Q}}
\def\TC{{\mathcal T}}
\def\ZC{{\mathcal Z}}
\def\AC{{\mathcal A}}
\def\CC{{\mathcal C}}
\def\DC{{\mathcal D}}
\def\EC{{\mathcal E}}
\def\FC{{\mathcal F}}
\def\GC{{\mathcal G}}
\def\HC{{\mathcal H}}
\def\IC{{\mathcal I}}
\def\JC{{\mathcal J}}
\def\KC{{\mathcal K}}
\def\LC{{\mathcal L}}
\def\MC{{\mathcal M}}
\def\NC{{\mathcal N}}
\def\OC{{\mathcal O}}
\def\PC{{\mathcal P}}
\def\UC{{\mathcal U}}
\def\VC{{\mathcal V}}
\def\XC{{\mathcal X}}
\def\SC{{\mathcal S}}
\def\RC{{\mathcal R}}

\def\BF{{\mathfrak B}}
\def\AF{{\mathfrak A}}
\def\GF{{\mathfrak G}}
\def\EF{{\mathfrak E}}
\def\CF{{\mathfrak C}}
\def\DF{{\mathfrak D}}
\def\JF{{\mathfrak J}}
\def\LF{{\mathfrak L}}
\def\MF{{\mathfrak M}}
\def\NF{{\mathfrak N}}
\def\XF{{\mathfrak X}}
\def\UF{{\mathfrak U}}
\def\KF{{\mathfrak K}}
\def\FF{{\mathfrak F}}

\def \longmapright#1{\smash{\mathop{\longrightarrow}\limits^{#1}}}
\def \mapright#1{\smash{\mathop{\rightarrow}\limits^{#1}}}
\def \lexp#1#2{\kern \scriptspace \vphantom{#2}^{#1}\kern-\scriptspace#2}
\def \linf#1#2{\kern \scriptspace \vphantom{#2}_{#1}\kern-\scriptspace#2}
\def \linexp#1#2#3 {\kern \scriptspace{#3}_{#1}^{#2} \kern-\scriptspace #3}

\def \Sel {{\mathop{\mathrm{Sel}}\nolimits}}
\def \Ext{\mathop{\mathrm{Ext}}\nolimits}
\def \ad{\mathop{\mathrm{ad}}\nolimits}
\def \sh{\mathop{\mathrm{Sh}}\nolimits}
\def \irr{\mathop{\mathrm{Irr}}\nolimits}
\def \FH{\mathop{\mathrm{FH}}\nolimits}
\def \FPH{\mathop{\mathrm{FPH}}\nolimits}
\def \coh{\mathop{\mathrm{Coh}}\nolimits}
\def \res{\mathop{\mathrm{Res}}\nolimits}
\def \op{\mathop{\mathrm{op}}\nolimits}
\def \rec {\mathop{\mathrm{rec}}\nolimits}
\def \art{\mathop{\mathrm{Art}}\nolimits}
\def \vol {\mathop{\mathrm{vol}}\nolimits}
\def \cusp {\mathop{\mathrm{Cusp}}\nolimits}
\def \scusp {\mathop{\mathrm{Scusp}}\nolimits}
\def \Iw {\mathop{\mathrm{Iw}}\nolimits}
\def \JL {\mathop{\mathrm{JL}}\nolimits}
\def \speh {\mathop{\mathrm{Speh}}\nolimits}
\def \isom {\mathop{\mathrm{Isom}}\nolimits}
\def \Vect {\mathop{\mathrm{Vect}}\nolimits}
\def \groth {\mathop{\mathrm{Groth}}\nolimits}
\def \hom {\mathop{\mathrm{Hom}}\nolimits}
\def \deg {\mathop{\mathrm{deg}}\nolimits}
\def \val {\mathop{\mathrm{val}}\nolimits}
\def \det {\mathop{\mathrm{det}}\nolimits}
\def \rep {\mathop{\mathrm{Rep}}\nolimits}
\def \spec {\mathop{\mathrm{Spec}}\nolimits}
\def \fr {\mathop{\mathrm{Fr}}\nolimits}
\def \frob {\mathop{\mathrm{Frob}}\nolimits}
\def \ker {\mathop{\mathrm{Ker}}\nolimits}
\def \im {\mathop{\mathrm{Im}}\nolimits}
\def \Red {\mathop{\mathrm{Red}}\nolimits}
\def \red {\mathop{\mathrm{red}}\nolimits}
\def \aut {\mathop{\mathrm{Aut}}\nolimits}
\def \diag {\mathop{\mathrm{diag}}\nolimits}
\def \spf {\mathop{\mathrm{Spf}}\nolimits}
\def \Def {\mathop{\mathrm{Def}}\nolimits}
\def \nrd {\mathop{\mathrm{nrd}}\nolimits}
\def \supp {\mathop{\mathrm{Supp}}\nolimits}
\def \Id {{\mathop{\mathrm{Id}}\nolimits}}
\def \lie {{\mathop{\mathrm{Lie}}\nolimits}}
\def \Ind{\mathop{\mathrm{Ind}}\nolimits}
\def \ind {\mathop{\mathrm{ind}}\nolimits}
\def \bad {\mathop{\mathrm{Bad}}\nolimits}
\def \top {\mathop{\mathrm{Top}}\nolimits}
\def \ker {\mathop{\mathrm{Ker}}\nolimits}
\def \coker {\mathop{\mathrm{Coker}}\nolimits}
\def \gal {{\mathop{\mathrm{Gal}}\nolimits}}
\def \Nr {{\mathop{\mathrm{Nr}}\nolimits}}
\def \rn {{\mathop{\mathrm{rn}}\nolimits}}
\def \tr {{\mathop{\mathrm{Tr~}}\nolimits}}
\def \Sp {{\mathop{\mathrm{Sp}}\nolimits}}
\def \st {{\mathop{\mathrm{St}}\nolimits}}
\def \sp{{\mathop{\mathrm{Sp}}\nolimits}}
\def \perv{\mathop{\mathrm{Perv}}\nolimits}
\def \tor {{\mathop{\mathrm{Tor}}\nolimits}}
\def \gr {{\mathop{\mathrm{gr}}\nolimits}}
\def \nilp {{\mathop{\mathrm{Nilp}}\nolimits}}
\def \obj {{\mathop{\mathrm{Obj}}\nolimits}}
\def \spl {{\mathop{\mathrm{Spl}}\nolimits}}
\def \unr {{\mathop{\mathrm{Unr}}\nolimits}}
\def \alg {{\mathop{\mathrm{Alg}}\nolimits}}
\def \grr {{\mathop{\mathrm{grr}}\nolimits}}
\def \cogr {{\mathop{\mathrm{cogr}}\nolimits}}
\def \coFil {{\mathop{\mathrm{coFil}}\nolimits}}

\def \rem{{\noindent\textit{Remark:~}}}
\def \rems{{\noindent\textit{Remarques:~}}}
\def \ext {{\mathop{\mathrm{Ext}}\nolimits}}
\def \End {{\mathop{\mathrm{End}}\nolimits}}

\def\semi{\mathrel{>\!\!\!\triangleleft}}
\let \DS=\displaystyle
\def\HT{{\mathop{\mathcal{HT}}\nolimits}}

\def \hi{\HC}
\newcommand*{\tarrow}{\relbar\joinrel\mid\joinrel\twoheadrightarrow}
\newcommand*{\harrow}{\lhook\joinrel\relbar\joinrel\mid\joinrel\rightarrow}
\newcommand*{\rarrow}{\relbar\joinrel\mid\joinrel\rightarrow}
\def \coim {{\mathop{\mathrm{Coim}}\nolimits}}
\def \can {{\mathop{\mathrm{can}}\nolimits}}
\def\LFF{{\mathscr L}}

\setcounter{secnumdepth}{3} \setcounter{tocdepth}{3}

\def \Fil{\mathop{\mathrm{Fil}}\nolimits}
\def \CoFil{\mathop{\mathrm{CoFil}}\nolimits}
\def \Fill{\mathop{\mathrm{Fill}}\nolimits}
\def \CoFill{\mathop{\mathrm{CoFill}}\nolimits}
\def\SF{{\mathfrak S}}
\def\PF{{\mathfrak P}}
\def \EFil{\mathop{\mathrm{EFil}}\nolimits}
\def \EFill{\mathop{\mathrm{EFill}}\nolimits}
\def \FP{\mathop{\mathrm{FP}}\nolimits}

\let \longto=\longrightarrow
\let \oo=\infty

\let \d=\delta
\let \k=\kappa

%\newcounter{num}
\renewcommand{\theequation}{\arabic{section}.\arabic{subsection}.\arabic{thm}}
\newcommand{\marque}{\addtocounter{thm}{1}
{\smallskip \noindent \textit{\thethm}~---~}}

\renewcommand\atop[2]{\ensuremath{\genfrac..{0pt}{1}{#1}{#2}}}

\newcommand\atopp[2]{\genfrac{}{}{0pt}{}{#1}{#2}}

\title[Construction of torsion cohomology classes]{Construction of torsion cohomology classes for KHT Shimura varieties}

%\alttitle{Cohomology of Lubin-Tate spaces: a new natural geometric proof}

\author{Boyer Pascal}
\email{boyer@math.univ-paris13.fr}
\address{Universit\'e Paris 13, Sorbonne Paris Cit\'e \\
LAGA, CNRS, UMR 7539\\ 
F-93430, Villetaneuse (France) \\
PerCoLaTor: ANR-14-CE25}

\frontmatter

\begin{abstract}
Let $\sh_K(G,\mu)$ be a  Shimura variety of KHT type,  as introduced in \cite{h-t}, associated 
to some similitude group $G/\Qm$ and a open compact subgroup $K$ of $G(\Am)$. For any irreducible
algebraic $\overline \Qm_l$-representation $\xi$ of $G$, let $V_\xi$ be the $\Zm_l$-local system 
on $\sh_K(G,\mu)$. From \cite{boyer-stabilization}, we know that if we allow the local component
$K_l$ of $K$ to be small enough, then there must exists some non trivial cohomology
classes with coefficient in $V_\xi$. The aim of this paper is then to construct explicitly such torsion
classes with the control of $K_l$. As an application we obtain the construction of some new automorphic
congruences between tempered and non tempered automorphic representations of the same
weight and same level at $l$.

\end{abstract}

%% Classification mathÈmatique  (2010)
%\subjclass{11G18, 11G10, 14G35, 14G22, 11F70, 11F80}

%\keywords{VariÈtÈs de Shimura, cohomologie de torsion, idÈal maximal de l'algËbre de Hecke, 
%localisation de la cohomologie, reprÈsentation galoisienne}
%% Mots et expressions clÈs en anglais :


\maketitle

\pagestyle{headings} \pagenumbering{arabic}

\tableofcontents
%
%\mainmatter
%
%\renewcommand{\theequation}{\arabic{section}.\arabic{subsection}.\arabic{smfthm}}

\section{Introduction}
\renewcommand{\theequation}{\arabic{equation}}
%\backmatter

The class number formula for number fields (resp. the Birch-Swinnerton-Dyer conjecture)
asserts that the order of vanishing of the Dedekind zeta function at $s=0$ of a number field $K$ 
(resp. the order of vanishing at $s=1$ of the $L$-function of some elliptic curve $E$ over a
number field $K$) with the rank of its group of units (resp. with the rank of the Mordell-Weil
group $E(K)$). Both of these statements can be restated in terms of rank of Selmer groups and is
generalized for $p$-adic motivic Galois representations in the Bloch-Kato conjecture.

Since the work of Ribet, one strategy to realize a part of this conjecture is to consider some
automorphic tempered representation $\Pi$ of a reductive group $G/\Qm$ and take a prime
divisor $l$ of some special values of its $L$-function. We try then to construct an automorphic
non tempered representation $\Pi'$ of $G$ congruent to $\Pi$ modulo $l$ in some sense
so that such an automorphic congruence produces a non trivial element in some Selmer group.

In \cite{boyer-mrl} we show how to produce automorphic congruences from torsion classes
in the cohomology of Kottwitz-Harris-Taylor Shimura varieties associated to $G$, 
with coefficients in a local system
$\LC_\xi$ indexed by irreducible algebraic representations $\xi$, called weight, of $G(\Qm)$.
For example, see corollary 2.9 of \cite{boyer-mrl}, to each non trivial torsion cohomology class of level $I$,
we can associate an infinite collection of non isomorphic weakly congruent irreducible automorphic 
representations of the same weight and level but each of them being tempered.
In section 3 of \cite{boyer-mrl}, we obtained automorphic congruences between tempered and
non tempered automorphic representations but with distinct weights.
In \cite{boyer-stabilization}, using completed cohomology, we constructed automorphic congruences
between tempered and non tempered automorphic representations of the same weight but without any
control of their respective level at $l$ which might be an issue to construct then non trivial
element in some Selmer groups, cf. loc. cit.

Another way to interpret the computations of \cite{boyer-stabilization}, is to say that, whatever is the weight
$\xi$, if you take the level at $l$ small enough, then the cohomology groups of your KHT Shimura variety
with coefficients in $\LC_\xi$ can't be all free, there must exist some non trivial cohomology classes.
The main aim of this paper is then to find explicit conditions for the existence of non trivial
cohomology classes with coefficients in $\LC_\xi$, without playing with the level at $l$. 

As in previous work, we compute the cohomology groups of the Shimura variety $\sh(G,\mu)$
through the vanishing cycles spectral sequence, i.e. as the cohomology of the special
fiber of $\sh(G,\mu)$ at some place $v$ not dividing $l$, 
with coefficients in $\LC_\xi \otimes \Psi_v$ where $\Psi_v$
designates the perverse sheaf of vanishing cycles at $v$. In \cite{boyer-torsion} we explained
how to, using the Newton stratification of the special fiber of $\sh(G,\mu)$, construct a
$\overline \Zm_l$-filtration of $\Psi_v$ which graduates are some intermediate extensions of
the $\overline \Zm_l$-Harris-Taylor local systems constructed in \cite{h-t}. These local systems are
indexed by irreducible $\overline \Qm_l$-cuspidal entire representations $\pi_v$ of the linear group of
rank $g \leq d$, where $d$ is the dimension of $\sh(G,\mu)$: among the data is some lattice of 
the Steinberg representation $\st_t(\pi_v)$ with $tg \leq d$.
We then have a spectral sequence $E_1^{p,q}$ computing the $H^{p+q}(\sh_K(G,\mu),\LC_\xi)$:
up to translation we may suppose that $E_1^{p,q}=0$ for all $p<0$.
\begin{itemize}
\item The first idea to construct torsion could be to find some non trivial torsion classes in the $E_1$-page, i.e.
in the cohomology of the Harris-Taylor perverse sheaves. For example in \cite{boyer-aif} proposition 4.5.1,
we prove that if the modulo $l$ reduction of such $\pi_v$ is cuspidal but not supercuspidal, then, for
a well chosen level $K$, the cohomology groups of the associated Harris-Taylor perverse sheaves, 
can't be all free, so there is torsion on the $E_1$ page. Unfortunately it seems not so easy to prove
that such torsion cohomology classes remains on the $E_\oo$-page.

\item We can then try to produce torsion in the $E_2$ page by finding a map $d_1^{p,q}$ with
$$\xymatrix{
E_1^{p,q} \otimes_{\overline \Zm_l} \overline \Qm_l 
\ar[rr]^{d_1^{p,q} \otimes_{\overline \Zm_l} \overline \Qm_l} \ar@{->>}[d] & & E_1^{p+1,q}
\otimes_{\overline \Zm_l} \overline \Qm_l \\
Q \ar[rr]^\sim && Q' \ar@{^{(}->}[u]
}$$
such that the $\overline \Zm_l$-lattices of $Q$ and $Q'$ respectively induced by $E_1^{p,q}$
and $E_1^{p+1,q}$, are not isomorphic.
\end{itemize}
The idea to realize this last point, is to use the main result of \cite{boyer-local-ihara} where we describe
some of the lattices of the $\st_t(\pi_v)$ which appears as some data of the Harris-Taylor perverse
sheaves as graduates of the filtration of stratification of $\Psi_v$. These lattices verify some non
degeneracy persitence property in the following sense: the socle of their modulo $l$ reduction is irreducible
and non degenerate. In particular when the local system is concentrated in the supersingular
locus, which means with the previous notations that $\st_t(\pi_v)$ is a representation of
$GL_d(F_v)$, then this persitence of non degeneracy is also true for $E_1^{0,0}$ while all the
$E_1^{p,1-p}$ for $p>0$ containing $\st_t(\pi_v)$ are non trivially parabolically induced.
If we manage so that this non degenerate socle is cuspidal, necessary these induced lattices
don't satisfy the persitence of non degeneracy property, so that the $E_2$-page has non trivial
cohomology classes and it's then quite easy to prove that it remains at the
$E_\oo$-page.

In terms of $L$-functions, our assumption to realize this program, corresponds to the following.
Find an irreducible automorphic representation $\Pi$ of level $K$ such that its local $L$-factor at $p$
modulo $l$, has a pole at $s=1$.

%\mainmatter

\renewcommand{\theequation}{\arabic{section}.\arabic{subsection}.\arabic{thm}}

\section{Cohomology of Harris-Taylor perverse sheaves}

\subsection{Shimura varieties of Kottwitz-Harris-Taylor type}

\label{para-geo}

Let $F=F^+ E$ be a CM field where $E/\Qm$ is quadratic imaginary and $F^+/\Qm$
totally real with a fixed real embedding $\tau:F^+ \hookrightarrow \Rm$. For a place $v$ of $F$,
we will denote
\begin{itemize}
\item $F_v$ the completion of $F$ at $v$,

\item $\OC_v$ the ring of integers of $F_v$,

\item $\varpi_v$ a uniformizer,

\item $q_v$ the cardinal of the residual field $\kappa(v)=\OC_v/(\varpi_v)$.
\end{itemize}
Let $B$ be a division algebra with center $F$, of dimension $d^2$ such that at every place $x$ of $F$,
either $B_x$ is split or a local division algebra and suppose $B$ provided with an involution of
second kind $*$ such that $*_{|F}$ is the complexe conjugation. For any
$\beta \in B^{*=-1}$, denote $\sharp_\beta$ the involution $x \mapsto x^{\sharp_\beta}=\beta x^*
\beta^{-1}$ and $G/\Qm$ the group of similitudes, denoted $G_\tau$ in \cite{h-t}, defined for every
$\Qm$-algebra $R$ by 
$$
G(R)  \simeq   \{ (\lambda,g) \in R^\times \times (B^{op} \otimes_\Qm R)^\times  \hbox{ such that } 
gg^{\sharp_\beta}=\lambda \}
$$
with $B^{op}=B \otimes_{F,c} F$. 
If $x$ is a place of $\Qm$ split $x=yy^c$ in $E$ then 
\addtocounter{thm}{1}
\begin{equation} \label{eq-facteur-v}
G(\Qm_x) \simeq (B_y^{op})^\times \times \Qm_x^\times \simeq \Qm_x^\times \times
\prod_{z_i} (B_{z_i}^{op})^\times,
\end{equation}
where, identifying places of $F^+$ over $x$ with places of $F$ over $y$,
$x=\prod_i z_i$ in $F^+$.

\noindent \textbf{Convention}: for $x=yy^c$ a place of $\Qm$ split in $E$ and $z$ a place of $F$ over $y$
as before, we shall make throughout the text, the following abuse of notation by denoting 
$G(F_z)$ in place of the factor $(B_z^{op})^\times$ in the formula (\ref{eq-facteur-v}).
%as well as $$G(\Am^z):=G(\Am^x) \times \bigl (\Qm_x^\times \times \prod_{z_i \neq z} (B_{z_i}^{op})^\times \bigr ).$$

In \cite{h-t}, the author justify the existence of some $G$ like before such that moreover
\begin{itemize}
\item if $x$ is a place of $\Qm$ non split in $E$ then $G(\Qm_x)$ is quasi split;

\item the invariants of $G(\Rm)$ are $(1,d-1)$ for the embedding $\tau$ and $(0,d)$ for the others.
\end{itemize}

As in  \cite{h-t} bottom of page 90, a compact open subgroup $U$ of $G(\Am^\oo)$ is said 
\emph{small enough}
if there exists a place $x$ such that the projection from $U^v$ to $G(\Qm_x)$ does not contain any 
element of finite order except identity.

\begin{nota}
Denote $\IC$ the set of open compact subgroup small enough of $G(\Am^\oo)$.
For $I \in \IC$, write $X_{I,\eta} \longrightarrow \spec F$ the associated
Shimura variety of Kottwitz-Harris-Taylor type.
\end{nota}

\begin{defin} \label{defi-spl}
Define $\spl$ the set of  places $v$ of $F$ such that $p_v:=v_{|\Qm} \neq l$ is split in $E$ and
$B_v^\times \simeq GL_d(F_v)$.  For each $I \in \IC$, write
$\spl(I)$ the subset of $\spl$ of places which doesn't divide the level $I$.
\end{defin}

In the sequel, $v$ will denote a place of $F$ in $\spl$. For such a place $v$ 
the scheme $X_{I,\eta}$ has a projective model $X_{I,v}$ over $\spec \OC_v$
with special fiber $X_{I,s_v}$. For $I$ going through $\IC$, the projective system $(X_{I,v})_{I\in \IC}$ 
is naturally equipped with an action of $G(\Am^\oo) \times \Zm$ such that 
$w_v$ in the Weil group $W_v$ of $F_v$ acts by $-\deg (w_v) \in \Zm$,
where $\deg=\val \circ \art^{-1}$ and $\art^{-1}:W_v^{ab} \simeq F_v^\times$ is Artin's isomorphism
which sends geometric Frobenius to uniformizers.

\begin{notas} (see \cite{boyer-invent2} \S 1.3)
For $I \in \IC$, the Newton stratification of the geometric special fiber $X_{I,\bar s_v}$ is denoted
$$X_{I,\bar s_v}=:X^{\geq 1}_{I,\bar s_v} \supset X^{\geq 2}_{I,\bar s_v} \supset \cdots \supset 
X^{\geq d}_{I,\bar s_v}$$
where $X^{=h}_{I,\bar s_v}:=X^{\geq h}_{I,\bar s_v} - X^{\geq h+1}_{I,\bar s_v}$ is an affine 
scheme\footnote{see for example \cite{ito2}}, smooth of pure dimension $d-h$ built up by the geometric 
points whose connected part of its Barsotti-Tate group is of rank $h$.
For each $1 \leq h <d$, write
$$i_{h}:X^{\geq h}_{I,\bar s_v} \hookrightarrow X^{\geq 1}_{I,\bar s_v}, \quad
j^{\geq h}: X^{=h}_{I,\bar s_v} \hookrightarrow X^{\geq h}_{I,\bar s_v},$$
and $j^{=h}=i_h \circ j^{\geq h}$.
\end{notas}
%
%\rem recall from \cite{boyer-invent2} that for every $1 \leq h \leq d-1$, if, with the preceding convention,
%$$I_v \simeq K_v(n):=\ker \bigl ( GL_d(\OC_v) \twoheadrightarrow GL_d(\OC_v/\varpi_v^n) \bigr )$$
%with $n \geq 1$, then
%$$X^{=h}_{I,\bar s_v}=X^{=h}_{I,\bar s_v,1} \times_{P_{h,d-h}(\OC_v/\varpi_v^n)} 
%GL_d(\OC_v/\varpi_v^n)$$ 
%is geometrically induced from some particular subscheme $X^{=h}_{I,\bar s_v,1}$ provided with an 
%action of the standard parabolic subgroup $P_{h,d-h}$ with Levi $GL_h \times GL_{d-h}$ and where 
%the action of $GL_h(\OC_v/\varpi_v^n)$ factors through the determinant.

For $1 \leq h < d$, the Newton stratum $X_{\IC,\bar s}^{=h}$ is geometrically induced
under the action of the parabolic subgroup $P_{h,d-h}(\OC_v)$ in the sense where there
exists a closed subscheme $X_{I,\bar s,\overline{1_h}}^{=h}$ stabilized by the Hecke 
action of $P_{h,d-h}(\OC_v)$ and such that
$$X_{\IC,\bar s}^{=h} \simeq X_{\IC,\bar s,\overline{1_h}}^{=h} 
\times_{P_{h,d-h}(\OC_v)} GL_d(\OC_v).$$
For $a \in GL_d(F_v)/P_{h,d-h}(F_v)$, we denote $X_{I,\bar s,a}^{=h}$ the image of
$X_{\IC,\bar s,\overline{1_h}}^{=h}$ by $a$ and $X_{I,\bar s,a}^{\geq h}$ its closure in 
$X_{I,\bar s}^{\geq h}$: they are stable under 
$P_a(F_v):=aP_{h,d-h}(F_v)a^{-1}$.

\subsection{Harris-Taylor perverse sheaves}

From now on, we fix a prime number $l$ unramified in $E$ and suppose that for every place $v$ of $F$ considered
after, its restriction $v_{|\Qm}$ is not equal to $l$.
For a representation $\pi_v$ of $GL_d(F_v)$ and $n \in \frac{1}{2} \Zm$, set 
$\pi_v \{ n \}:= \pi_v \otimes q_v^{-n \val \circ \det}$. Recall that
the normalized induction of two representations $\pi_{v,1}$ and $\pi_{v,2}$ 
of respectively $GL_{n_1}(F_v)$ and $GL_{n_2}(F_v)$ is
$$\pi_1 \times \pi_2:=\ind_{P_{n_1,n_1+n_2}(F_v)}^{GL_{n_1+n_2}(F_v)}
\pi_{v,1} \{ \frac{n_2}{2} \} \otimes \pi_{v,2} \{-\frac{n_1}{2} \}.$$
A representation
$\pi_v$ of $GL_d(F_v)$ is called \emph{cuspidal} (resp. \emph{supercuspidal})
if it's not a subspace (resp. subquotient) of a proper parabolic induced representation.

\begin{defin} \label{defi-rep} (see \cite{zelevinski2} \S 9 and \cite{boyer-compositio} \S 1.4)
Let $g$ be a divisor of $d=sg$ and $\pi_v$ an irreducible cuspidal 
$\overline \Qm_l$-representation of $GL_g(F_v)$. Then the normalized induced representation 
$$\pi_v\{ \frac{1-s}{2} \} \times \pi_v \{ \frac{3-s}{2} \} \times \cdots \times \pi_v \{ \frac{s-1}{2} \}$$ 
holds a unique irreducible quotient (resp. subspace) denoted $\st_s(\pi_v)$ (resp.
$\speh_s(\pi_v)$); it's a generalized Steinberg (resp. Speh) representation.
\end{defin}

The local Jacquet-Langlands correspondance is a bijection between irreducible essentially square
integrable representations of $GL_d(F_v)$, i.e. representations of the type $\st_s(\pi_v)$
for $\pi_v$ cuspidal, and irreducible representations of $D_{v,d}^\times$ where
$D_{v,d}$ is the central division algebra over $F_v$ with invariant $\frac{1}{d}$ and with maximal
order $\DC_{v,d}$.

\begin{nota}
We will denote $\pi_v[s]_D$ the irreductible representation of $D_{v,d}^\times$ associated 
to $\st_s(\pi_v^\vee)$ by the local Jacquet-Langlands correspondance.
\end{nota}

Let $\pi_v$ be an irreducible cuspidal $\overline \Qm_l$-representation of $GL_g(F_v)$ and
fix $t \geq 1$ such that $tg \leq d$. Thanks to Igusa varieties, Harris and Taylor
constructed a local system on  $X^{=tg}_{\IC,\bar s,\overline{1_h}}$
$$\LC_{\overline \Qm_l}(\pi_v[t]_D)_{\overline{1_h}}=\bigoplus_{i=1}^{e_{\pi_v}} 
\LC_{\overline \Qm_l}(\rho_{v,i})_{\overline{1_h}}$$
where $(\pi_v[t]_D)_{|\DC_{v,h}^\times}=\bigoplus_{i=1}^{e_{\pi_v}} \rho_{v,i}$ 
with $\rho_{v,i}$ irreductible. The Hecke action of $P_{tg,d-tg}(F_v)$ is then given
through its quotient $GL_{d-tg} \times \Zm$. These local systems have stable
$\overline \Zm_l$-lattices and we will write simply $\LC(\pi_v[t]_D)_{\overline{1_h}}$
for any $\overline \Zm_l$-stable lattice that we don't want to specify.

\begin{notas} For $\Pi_t$ any representation of $GL_{tg}$ and 
$\Xi:\frac{1}{2} \Zm \longrightarrow \overline \Zm_l^\times$ defined by 
$\Xi(\frac{1}{2})=q^{1/2}$, we introduce
$$\widetilde{HT}_1(\pi_v,\Pi_t):=\LC(\pi_v[t]_D)_{\overline{1_h}} 
\otimes \Pi_t \otimes \Xi^{\frac{tg-d}{2}}$$
and its induced version
$$\widetilde{HT}(\pi_v,\Pi_t):=\Bigl ( \LC(\pi_v[t]_D)_{\overline{1_h}} 
\otimes \Pi_t \otimes \Xi^{\frac{tg-d}{2}} \Bigr) \times_{P_{tg,d-tg}(F_v)} GL_d(F_v),$$
where the unipotent radical of $P_{tg,d-tg}(F_v)$ acts trivially and the action of
$$(g^{\oo,v},\left ( \begin{array}{cc} g_v^c & *†\\ 0 & g_v^{et} \end{array} \right ),\sigma_v) 
\in G(\Am^{\oo,v}) \times P_{tg,d-tg}(F_v) \times W_v$$ 
is given
\begin{itemize}
\item by the action of $g_v^c$ on $\Pi_t$ and 
$\deg(\sigma_v) \in \Zm$ on $ \Xi^{\frac{tg-d}{2}}$, and

\item the action of $(g^{\oo,v},g_v^{et},\val(\det g_v^c)-\deg \sigma_v)
\in G(\Am^{\oo,v}) \times GL_{d-tg}(F_v) \times \Zm$ on $\LC_{\overline \Qm_l}
(\pi_v[t]_D)_{\overline{1_h}} \otimes \Xi^{\frac{tg-d}{2}}$.
\end{itemize}
We also introduce
$$HT(\pi_v,\Pi_t)_{\overline{1_h}}:=\widetilde{HT}(\pi_v,\Pi_t)_{\overline{1_h}}[d-tg],$$
and the perverse sheaf
$$P(t,\pi_v)_{\overline{1_h}}:=j^{=tg}_{\overline{1_h},!*} HT(\pi_v,\st_t(\pi_v))_{\overline{1_h}} 
\otimes \Lm(\pi_v),$$
and their induced version, $HT(\pi_v,\Pi_t)$ and $P(t,\pi_v)$, where 
%$$j^{=h}:=i^h \circ j^{\geq h}:X^{=h}_{\IC,\bar s} \hookrightarrow X^{\geq h}_{\IC,\bar s} \hookrightarrow X_{\IC,\bar s}$$ 
$\Lm^\vee$ is the local Langlands correspondence.
\end{notas}

\rem recall that $\pi'_v$ is said inertially equivalent to $\pi_v$ 
if there exists a character $\zeta: \Zm \longrightarrow  \overline \Qm_l^\times$ such that
$\pi'_v \simeq \pi_v \otimes (\zeta \circ \val \circ \det)$.
Note, cf. \cite{boyer-invent2} 2.1.4, that $P(t,\pi_v)$ depends only on the inertial class of $\pi_v$ and
$$P(t,\pi_v)=e_{\pi_v} \PC(t,\pi_v)$$ 
where $\PC(t,\pi_v)$ is an irreducible perverse sheaf. When we want to speak of
the $\overline \Qm_l$-versions we will add it on the notations.

Recall that the modulo $l$ reduction of an irreducible $\overline \Qm_l$-representation is still irreducible
and cuspidal but not necessary supercuspidal. In this last case, its supercuspidal support is a
Zelevinsky segment associated to some unique inertial equivalent classe of supercuspidal $\overline \Fm_l$-representation $\varrho$ of length in the following set
$\{ 1, m(\varrho),l m(\varrho),l^2 m(\varrho),\cdots \}.$

\begin{defi} \label{defi-type}
Let $i \in \Zm$ be greater than $-1$.
We say that $\pi_v$ is of $\varrho$-type $i$ when the supercuspidal support of its modulo
$l$ reduction is a Zelevinsky segment associated to $\varrho$ of length $1$ for $i=-1$ and 
$m(\varrho)l^i$ otherwise. 
\end{defi}

\begin{nota} We denote $\scusp_{\overline \Fm_l}(g)$ the set of inertial equivalence classes of
$\overline \Fm_l$-supercuspidal representations of $GL_g(F_v)$. For 
$\varrho \in \scusp_{\overline \Fm_l}(g)$, we then denote 
\begin{itemize}
\item $g_{-1}(\varrho)=g$ and for $i \geq 0$, $g_i(\varrho)=m(\varrho)l^i g_{-1}(\varrho)$;

\item $\scusp_i(\varrho)$ the set of inertial equivalence
classes of irreducible cuspidal $\overline \Qm_l$-representations of $\varrho$-type $i$.
\end{itemize}
\end{nota}

\subsection{Cohomology groups over $\overline \Qm_l$}

Let $\sigma_0:E \hookrightarrow
\overline{\Qm}_l$ be a fixed embedding and write $\Phi$ the set of embeddings 
$\sigma:F \hookrightarrow \overline \Qm_l$ whose restriction to $E$ equals $\sigma_0$.
There exists, \cite{h-t} p.97, then an explicit bijection between irreducible algebraic representations 
$\xi$ of $G$ over $\overline \Qm_l$ and $(d+1)$-uple
$\bigl ( a_0, (\overrightarrow{a_\sigma})_{\sigma \in \Phi} \bigr )$
where $a_0 \in \Zm$ and for all $\sigma \in \Phi$, we have $\overrightarrow{a_\sigma}=
(a_{\sigma,1} \leq \cdots \leq a_{\sigma,d} )$. 
We then denote $V_{\xi}$ the associated $\overline \Zm_l$-local system on $X_\IC$ 
Recall that an irreducible automorphic representation $\Pi$ is said $\xi$-cohomological if there exists
an integer $i$ such that
$$H^i \bigl ( ( \lie ~G(\Rm)) \otimes_\Rm \Cm,U,\Pi_\oo \otimes \xi^\vee \bigr ) \neq (0),$$
where $U$ is a maximal open compact subgroup modulo the center of $G(\Rm)$.

\begin{defin} \label{defi-degeneracy} (cf. \cite{M-W})
For $\Pi$ an automorphic irreducible representation $\xi$-cohomological of $G(\Am)$,
then, see for example lemma 3.2 of \cite{boyer-aif},  there exists an integer $s$ called the degeneracy depth of $\Pi$,
such that through Jacquel-Langlands correspondence and base change, its associated representation of $GL_d(\Am_\Qm)$
is isobaric of the following form
$$\mu | \det |^{\frac{1-s}{2}} \boxplus \mu | \det |^{\frac{3-s}{2}} \boxplus \cdots \boxplus \mu | \det |^{\frac{s-1}{2}}$$
where $\mu$ is an irreducible cuspidal representation of $GL_{d/s}(\Am_\Qm)$.
\end{defin}

\rem for a place $v$ such that $G(F_v) \simeq GL_d(F_v)$ in the sense of our previous convention,
the local component $\Pi_v$ of $\Pi$ at $v$ is isomorphic to some $\speh_s(\pi_v)$ where $\pi_v$
is an irreducible non degenerate representation, $s \geq 1$ is an integer and $\speh_s(\pi_v)$
is the Langlands quotients of the parabolic induced representation $\pi_v \{ \frac{1-s}{2} \} \times 
\pi_v \{ \frac{3-s}{2} \} \times \cdots \times \pi_v \{ \frac{s-1}{2} \}$. In terms of the Langlands
correspondence, $\speh_s(\pi_v)$ corresponds to $\sigma \oplus \sigma(1) \oplus \cdots \oplus \sigma(s-1)$ 
where $\sigma$ is the representation of $\gal(\bar F/F)$ associated to $\pi_v$ by the local Langlands 
correspondence.

\begin{nota}
For $\pi_v$ an irreducible cuspidal $\overline \Qm_l$-representation of $GL_g(F_v)$ and $t \geq 1$
such that $tg \leq d$, let denote
$$H^i_{!,1}(\pi_v,t,\xi):=\lim_{\atopp{\longrightarrow}{I \in \IC_v}}
H^i_c(X_{I,\bar s_v,\overline{1_{tg}}}^{=tg},V_\xi \otimes
j^{=tg,*}_{\overline{1_{tg}}} \PC(\pi_v,t)_{\overline{1_{tg}}})$$
and the induced version
$$H^i_{!}(\pi_v,t):=\lim_{\atopp{\longrightarrow}{I \in \IC_v}}
H^i_c(X_{I,\bar s_v}^{\geq tg},V_\xi \otimes
j^{=tg,*} \PC(\pi_v,t)) \simeq H^i_{!,1}(\pi_v,t)\times_{P_{tg,d-tg}(F_v)} GL_d(F_v),$$
as well as
$$H^i_{!*,1}(\pi_v,t):=\lim_{\atopp{\longrightarrow}{I \in \IC_v}}
H^i(X_{I,\bar s_v,\overline{1_{tg}}}^{\geq tg},V_\xi \otimes \PC(\pi_v,t)_{\overline{1_{tg}}})$$
and
$$H^i(\pi_v,t)_{!*}:=\lim_{\atopp{\longrightarrow}{I \in \IC_v}}
H^i(X_{I,\bar s_v}^{\geq tg},V_\xi \otimes \PC(\pi_v,t)) \simeq 
H^i_{!*,1}(\pi_v,t)\times_{P_{tg,d-tg}(F_v)} GL_d(F_v),$$
where $\IC_v$ is the set of $I \in \IC$ such that $I_v$ is of the form $K_v(n)$ for some $n \geq 1$.
\end{nota}

In this section we only consider the $\overline \Qm_l$-cohomology groups and we recall the computations
of \cite{boyer-compositio}. For $\Pi^{\oo,v}$ an irreducible representation of $G(\Am^{\oo,v})$, 
denote $[H^i_!(\pi_v,t,\xi) ] \{ \Pi^{\oo,v} \}$ the associated isotypic component. We will use similar
notations with the cohomology groups introduced above.
Consider now a fixed irreducible cuspidal representation $\pi_v$ of $GL_g(F_v)$.

\begin{propo} \label{prop-coho-Ql} (cf. \cite{boyer-aif} \S 3.2 and 3.3)
Let $\Pi$ be an irreducible automorphic representation of $G(\Am)$ which is $\xi$-cohomological and
with degeneracy depth $s \geq 1$.
\begin{itemize}
\item If $s=1$ then $[H^i_!(\pi_v,t,\xi) ] \{ \Pi^{\oo,v} \}$ and
$[H^i_{!,*}(\pi_v,t,\xi) ] \{ \Pi^{\oo,v} \}$ are all zero for $i \neq 0$.
For $i=0$, if $[H^i_!(\pi_v,t,\xi) ] \{ \Pi^{\oo,v} \} \neq (0)$ (resp.
$[H^i_{!*}(\pi_v,t,\xi) ] \{ \Pi^{\oo,v} \} \neq (0)$) then $\Pi_v$ is of the following shape
$\st_k(\widetilde \pi_v) \times \Pi'_v$ where $\widetilde \pi_v$ is inertially equivalent to $\pi_v$
and $k \leq t$ (resp. $k=t$).

\item For $s \geq 1$, and $\Pi_v$ of the following shape $\speh_s(\pi_v \times \pi'_v)$,
$\pi'_v$ any irreducible representation of $GL_{\frac{d-sg}{s}}(F_v)$, then
$[H^i_!(\pi_v,t,\xi) ] \{ \Pi^{\oo,v} \}$ (resp. $[H^i_{!*}(\pi_v,t,\xi) ] \{ \Pi^{\oo,v} \}$) is non zero if and only
if $i=s-1$ and $t \geq s$ (resp. $t=s$ and $i \equiv s-1 \mod 2$ with $|i| \leq s-1$).
\end{itemize}
\end{propo}

\rem In \cite{boyer-aif}, we give the complete description of these cohomology groups.

\section{Torsion in the cohomology of KHT-Shimura varieties}

As explained in the introduction, to construct non trivial torsion cohomology classes with coefficients
in $\LC_\xi$, we use a spectral sequence $E_1^{p,q}$ obtained from a filtration of stratification
of the perverse sheaf of nearby cycles $\Psi_\IC$, whose $E_1$ terms are given by the cohomology
groups of Harris-Taylor perverse sheaves. The argument takes place in two steps: 
\begin{itemize}
\item the construction of non trivial torsion classes in the $E_2$ page, cf. \S \ref{para-torsion-classes1};

\item we then have to prove that these previous torsion classes remain in the $E_\oo$.
\end{itemize}
To be able to do the second point we need to have some informal information about the
torsion classes in the cohomology of the Harris-Taylor perverse sheaves: it's the aim of the next
section.

\subsection{Torsion for Harris-Taylor perverse sheaves}

From now on, we fix an irreducible supercuspidal $\overline \Fm_l$-representation $\varrho$ and all the
irreducible cuspidal $\overline \Qm_l$-representation $\pi_v$ considered will be of type $\varrho$.
In \cite{boyer-torsion}, using the adjunction maps $\Id \longrightarrow j^{=h}_* j^{=h,*}$, we construct
a filtration called of stratification in loc. cit.
$$0=\Fil^{-d}_*(\pi_v,\Pi_t) \subset \Fil^{1-d}_*(\pi_v,\Pi_t) \subset \cdots \subset
\Fil^0_*(\pi_v,\Pi_t)=j^{=tg}_{!} HT(\pi_v, \Pi_t),$$
with free gradutates $\gr^{-r}_*(\pi_v,\Pi_t):=\Fil^{-r}_*(\pi_v,\Pi_t)/\Fil^{-r-1}_*(\pi_v,\Pi_t)$ which are trival
except for $r=kg-1$ with $t \leq k \leq s$ and then verifying
\begin{multline*}
\lexp p  j^{= kg}_{!*} HT (\pi_v,\Pi_t \overrightarrow{\times} \st_{k-t}(\pi_v)) \otimes \Xi^{(t-k)/2} \\
\htarrow \gr^{1-kg}_*(\pi_v,\Pi_t) \htarrow  \\
\lexp {p+}  j^{= kg}_{!*} 
HT (\pi_v,\Pi_t \overrightarrow{\times} \st_{k-t}(\pi_v)) \otimes \Xi^{(t-k)/2},
\end{multline*}
where we recall that $\htarrow_+$ means a bimorphism, i.e. both a mono and a epi-morphism,
whose cokernel has support in $X^{\geq (kg+1}_{\IC,\bar s}$.
In \cite{boyer-duke}, we in fact proved that each of these graduates are isomorphic to the $p$-intermediate
extensions.

Meanwhile when $g=1$ and $\pi_v=\chi_v$ is a character, 
then the associated Harris-Taylor local system on $X^{=h}_{\IC,\bar s}$ is just the trivial one
$\overline \Zm_l$ where the fundamental group $\Pi_1(X^{=h}_{\IC,\bar s})$ 
acts by its quotient $\Pi_1(X^{=h}_{\IC,\bar s}) \twoheadrightarrow \DC_{v,h}^\times$ with
$\DC_{v,h}^\times$ acting by the character $\chi_v$. Then as $X^{\geq h}_{\IC,\bar s,\overline{1_h}}$
is smooth over $\spec  \overline \Fm_p$, then this Harris-Taylor local system shifted by the dimension
$d-h$, is perverse for both $t$-structures $p$ and $p+$, in particular the two intermediate extensions are 
equal so the previous short exact sequence is trivially true when $g=1$.

One of the main result of \cite{boyer-duke} is that this equality of perverse extensions remains true
for every Harris-Taylor local systems associated to any irreducible cuspidal representation $\pi_v$ such
that its modulo $l$ reduction is still supercuspidal, that is, with definition \ref{defi-type}, is of $\varrho$-type
$-1$.

More generally for every $\pi_v \in \scusp_i(\varrho)$ we prove the following long exact sequence
\addtocounter{thm}{1}
\begin{multline} \label{eq-resolution0}
0 \rightarrow j_!^{=s g} HT(\pi_v,\Pi_t \overrightarrow{\times} \speh_{s-t}(\pi_v)) \otimes \Xi^{\frac{s-t}{2}}
\longrightarrow \cdots \longrightarrow \\
j^{=(t+2)g}_! HT(\pi_v,\Pi_t†\overrightarrow{\times} \speh_2(\pi_v)) \otimes \Xi^1 
\longrightarrow j_!^{= (t+1)g} HT(\pi_v,\Pi_t \overrightarrow{\times} \pi_v) \otimes \Xi^{\frac{1}{2}} \\
\longrightarrow  
j^{=tg}_! HT(\pi_v,\Pi_t) \longrightarrow \lexp p j^{=tg}_{!*} HT(\pi_v,\Pi_t)
\rightarrow 0,
\end{multline}
which is equivalent to the property that the sheaf cohomology groups of $\lexp p j^{=tg}_{!*} HT(\pi_v,\Pi_t)$
are torsion free.

For $I \in \IC$ a finite level, let $\Tm_I:=\prod_{x \in \unr(I)} \Tm_{x}$ be the unramified Hecke algebra
where $\unr(I)$ is the union of places $x$ where $G$ is unramified and $I_x$ maximal, and where
$\Tm_x \simeq \overline \Zm_l[X^{un}(T_x)]^{W_x}$ for $T_x$ a split torus,
$W_x$ the spherical Weyl group and $X^{un}(T_x)$ the set of $\overline \Zm_l$-unramified characters 
of $T_x$. 

\begin{nota}
Consider a fixed maximal ideal $\mathfrak m$ of $\Tm_I$. For every $q \in \unr(I)$ let denote
$S_{\mathfrak m}(q)$ the multi-set %\footnote{A multi-set is a set with multiplicities.} 
of modulo $l$ Satake parameters at $q$ associated to $\mathfrak m$. For any $\Tm_I$-module $M$,
we moreover denote $M_{\mathfrak m}$ its localization at $\mathfrak m$.
\end{nota}

Consider $\pi \in \scusp_{-1}(\varrho)$ and denote $g=g_{-1}(\varrho)$ and let denote
$$H^i_{!*}(\pi_v,t)_{\mathfrak m}:=\lim_{\atopp{\longrightarrow}{I \in \IC_v}}
H^i(X_{I,\bar s_v}^{\geq tg},V_\xi \otimes 
\lexp p j^{=tg}_{!*} HT(\pi_v,t))_{\mathfrak m}$$
and
$$H^i_{!}(\pi_v,t)_{\mathfrak m}:=\lim_{\atopp{\longrightarrow}{I \in \IC_v}}
H^i(X_{I,\bar s_v}^{\geq tg},V_\xi \otimes j^{=tg}_{!} HT(\pi_v,t))_{\mathfrak m}.$$
In \cite{boyer-mrl} we proved that torsion classes arising in some cohomology group of the whole Shimura
variety, can be raised in characteristic zero to some automorphic tempered representation of $G(\Am)$ in 
the following sense.

\begin{defi}
A torsion class either in $H^i_{!*}(\pi_v,t)_{\mathfrak m}$ or in $H^i_{!}(\pi_v,t)_{\mathfrak m}$,
is said tempered $\xi$-cohomologial if there exists
an irreducible automorphic and $\xi$-cohomological tempered representation $\Pi$ unramified outside
$I$ and $p$ with $\Pi^\oo$ a subquotient of $\lim_{\rightarrow I \in \IC} 
H^{d-1}(X_{I,\bar \eta},V_{\xi,\overline \Qm_l})_{\mathfrak m}$.
\end{defi}

From now on we moreover suppose that $d=g_{-1}(\varrho)m(\varrho)l^u$ and we will pay attention to
irreducible $GL_d(F_v)$-subquotient of either $H^i_{!*}(\pi_v,t)_{\mathfrak m}[l]$ or 
$H^i_{!}(\pi_v,t)_{\mathfrak m}[l]$, isomorphic to $\rho_u$.

\begin{lemm} \label{lem-torsion-rel}
With the previous notations and assumptions with $\pi_{v,i} \in \scusp_{i}(\varrho)$ for $i \geq -1$, 
if there exists an irreducible subquotient
of $H^j_{!*}(\pi_{v,i},t)_{\mathfrak m}[l]$ (resp. $H^j_{!}(\pi_{v,i},t)_{\mathfrak m}[l]$), 
which is $GL_d(F_v)$-isotypic for $\rho_u$, then $j \in \{ 0,1 \}$ (resp. $j=1$).
\end{lemm}

\begin{proof} 
Consider first the case of $i=-1$.
We argue by induction from $t=s=m(\varrho)l^u$ to $t=1$ with both $H^i_{!,1}(\pi_{v,-1},t)_{\mathfrak m}$
and $H^i_{!*}(\pi_{v,-1},t)_{\mathfrak m}$. Concerning $H^i_{!*}(\pi_{v,-1},t)_{\mathfrak m}$, 
recall that, as $\pi_{v,-1} \in \scusp_{-1}(\varrho)$ so that 
$$\lexp p j^{=tg}_{!*} HT(\pi_{v,-1},t)  \simeq \lexp {p+} j^{=tg}_{!*} HT(\pi_{v,-1},t),$$
then we only have to consider the case $i \leq 0$.
By Artin's theorem, see for example theorem 4.1.1 of \cite{BBD}, using the affiness of $X^{=h}_{I,\bar s_v}$, 
we know that $H^i_{!}(\pi_{v,-1},t)_{\mathfrak m}$ is zero for every $i<0$ and without torsion for $i=0$.

Note first that for $t=s$, then $HT(\pi_{v,-1},s)$ has support in dimension zero, so that 
$H^i_{!*}(\pi_{v,-1},s)_{\mathfrak m}=H^i_{!}(\pi_{v,-1},s)_{\mathfrak m}$ is zero for $i \neq 0$ and free for $i=0$, 
so the result is trivially true. Suppose by induction, the result true for all $t>t_0$ and consider the case of 
$H^i_{!*}(\pi_{v,-1},t)_{\mathfrak m}$ through the spectral sequence associated to the 
resolution (\ref{eq-resolution0}). Note first that concerning irreducible subquotient of the $l$-torsion
of the cohomology groups which are $GL_d(F_v)$-isomorphic to $\rho_u$, then we can truncate
(\ref{eq-resolution0}) to the short exact sequence of its last three terms.
\addtocounter{thm}{1}
\begin{multline} \label{eq-resolution1}
0 \dashrightarrow j_!^{= (t+1)g} HT(\pi_{v,-1},\Pi_t \overrightarrow{\times} \pi_{v,-1}) \otimes \Xi^{\frac{1}{2}} 
\longrightarrow  \\
j^{=tg}_! HT(\pi_{v,-1},\Pi_t) \longrightarrow \lexp p j^{=tg}_{!*} HT(\pi_{v,-1},\Pi_t)
\rightarrow 0.
\end{multline}
Then considering our problem for $H^i_{!*}(\pi_{v,-1},t)_{\mathfrak m}$, there is nothing to prove for $i \leq -1$
and for $i=0$ we conclude because over $\overline \Qm_l$
$$H^0(X_{I,\bar s_v}, j_!^{= (t+1)g} HT(\pi_{v,-1},\Pi_t \overrightarrow{\times} \pi_{v,-1}) \otimes \Xi^{\frac{1}{2}})
\longrightarrow H^0(X_{I,\bar s_v},j^{=tg}_! HT(\pi_{v,-1},\Pi_t))$$
is related to tempered automorphic $\xi$-cohomological representations.
We are then done with $H^i_{!*}(\pi_{v,-1},t)_{\mathfrak m}$.
The result about $H^i_{!}(\pi_{v,-1},t)_{\mathfrak m}$, then follows from the long exact sequence associated
to (\ref{eq-resolution1}).
%
% we have to look at the non trivial maps over $\overline \Qm_l$
%$$H^i(X_{I,\bar s_v},j^{=tg}_{!*} HT(\pi_{v,-1},\Pi_t)) \longrightarrow
%H^{i+1}(X_{I,\bar s_v}, j_!^{= (t+1)g} HT(\pi_{v,-1},\Pi_t \overrightarrow{\times} \pi_{v,-1}) \otimes \Xi^{\frac{1}{2}}),$$
%concerning non tempered representations. Note again that we are only interested to $\rho_u$ so that
%the only case where it can appeared in the modulo $l$ reduction of some non tempered irreducible
%subquotient of $H^i(X_{I,\bar s_v},j^{=tg}_{!*} HT(\pi_{v,-1},\Pi_t))$ is when either $(t,i)=(s-1,\pm 1)$
%or $(t,i)=(s-2,0)$ which gives the expected result.

\medskip

Consider now the case $i \geq 0$. Recall, cf. \cite{dat-jl} proposition 2.3.3, that the semi-simplification of
the modulo $l$ reduction of $\pi_{v,i}[t]_D$, doesn't depend of the choice of a stable lattice, and is equal to
$$\sum_{k=0}^{m(\varrho)l^i -1}\tau\{ - \frac{m(\varrho)l^i-1}{2} +k\}$$ 
where $\tau$ is the modulo $l$ reduction of $\pi_{v,-1}[tm(\varrho)l^i]_D$ which is irreducible, and
$\tau \{ n \}:=\tau \otimes q^{-n \val \circ \nrd}$ where $\nrd$ is the reduced norm.
In particular for any representation $\Pi_t$ of $GL_{tg_{-1}(\varrho)}(F_v)$, we have
\addtocounter{smfthm}{1}
\begin{multline} \label{eq-F-chgt2}
m(\varrho)l^i \Fm \Bigl [  j_!^{= tm(\varrho)l^ig_{-1}(\varrho)} HT(\pi_{v,-1},\Pi)\Bigr ] =m(\varrho)l^{i}
j_!^{= tm(\varrho)l^ig_{-1}(\varrho)} \Bigl [ \Fm HT(\pi_{v,-1},\Pi) \Bigr ]  \\ = 
j_!^{= tg_i(\varrho)} \Bigl [ \Fm HT(\pi_{v,i},\Pi) \Bigr ] = \Fm \Bigl [ 
j_!^{= tg_i(\varrho)} HT(\pi_{v,i},\Pi) \Bigr ] .$$
\end{multline}
Note moreover that concerning subquotient isomorphic to $\rho_u$,
the only case where it can appeared in the modulo $l$ reduction of some irreducible
subquotient of the free quotient of $H^j(X_{I,\bar s_v},j^{=tg}_{!} HT(\pi_{v,-1},\Pi_t))$ is when 
either $(t,j)=(s-1,1)$ or $j=0$.
The result about $H^j_{!}(\pi_{v,i},t)_{\mathfrak m}[l]$ 
then follows from the previous case where $i=-1$ using (\ref{eq-F-chgt2}) and the following 
wellknown short exact sequence
$$0 \to H^n(X,\PC) \otimes_{\Zm_l} \Fm_l \longrightarrow H^n(X,\Fm \PC ) 
\longrightarrow H^{n+1}(X,\PC) [l] \to 0,$$
for any $\Fm_q$-scheme $X$ and any $\Zm_l$-perverse free sheaf $\PC$.

Then the result about the cohomology of $\lexp p j^{=tg}_{!*} HT(\pi_{v,i},\Pi_t)$ follows from
the resolution analog of (\ref{eq-resolution1}), and the case of $\lexp {p+} j^{=tg}_{!*} HT(\pi_{v,i},\Pi_t)$
is obtained by Grothendieck-Verdier duality.
\end{proof}

\subsection{Perverse sheaf of vanishing cycles and lattices}

\begin{nota}
For $I \in \IC$, let 
$$\Psi_{I,v}:=R\Psi_{\eta_v,I}(\overline \Zm_l[d-1])(\frac{d-1}{2})$$
be the vanishing cycle autodual perverse sheaf on $X_{I,\bar s}$.
\end{nota}

In \cite{boyer-duke},  we gave a decomposition
$$\Psi_\IC \simeq \bigoplus_{g=1}^d \bigoplus_{\varrho \in \scusp_{\overline \Fm_l}(g)} \Psi_{\varrho}$$
with $\Psi_\varrho \otimes_{\overline \Zm_l} \overline \Qm_l \simeq 
\bigoplus_{\pi_v \in \cusp(\varrho)} \Psi_{\pi_v}$ where
the irreducible constituant of 
$\Psi_{\pi_v}$ are exactly the
perverse Harris-Taylor sheaves attached to $\pi_v$.

Recall now from \cite{boyer-torsion}, how to construct filtrations of a free perverse sheaf such that its
graduates are free. Start first from an open subscheme $j:U \hookrightarrow X$ and 
$i: F:=X \backslash U \hookrightarrow X$, such that $j$ is affine. We then have the following adjunction morphism
where we use the notation $L \htarrow L'$ (resp.   $L \htarrow_+ L'$) for a bimorphism, i.e.
both a monomorphism and a epimorphism (resp. such that its cokernel is of dimension 
strictly less than those of the support of $L$):
$$\xymatrix{ 
& L \ar[drr]^{\can_{*,L}} \\
\lexp {p+} j_! j^* L \ar[ur]^{\can_{!,L}} 
\ar@{->>}[r]|-{+} & \lexp p j_{!*}j^* L \ar@{^{(}->>}[r]_+ & 
\lexp {p+} j_{!*} j^* L \ar@{^{(}->}[r]|-{+} & \lexp p j_*j^* L
}$$
where below is, cf. the remark following 1.3.12 de \cite{boyer-torsion},
the canonical factorisation of $\lexp {p+} j_! j^* L \longrightarrow \lexp {p} j_* j^* L$
and where the maps $\can_{!,L}$ and $\can_{*,L}$ are given by the adjonction property.

\begin{nota} \label{nota-filtration0} (cf. lemma 2.1.2 of \cite{boyer-torsion})
We introduce the filtration $\Fil^{-1}_{U,!}(L) \subset \Fil^0_{U,!}(L) \subset L$ with
$$\Fil^0_{U,!}(L)=\im_\FC (\can_{!,L}) \quad \hbox{ and } \quad \Fil^{-1}_{U,!}(L)=
\im_\FC \Bigl ( (\can_{!,L})_{|\PC_L} \Bigr ),$$
where $\PC_L:=i_*\lexp p \hi^{-1}_{libre} i^*j_* j^* L$ is the kernel of
$\ker_\FC \Bigl ( \lexp {p+} j_!  j^* L \twoheadrightarrow \lexp p j_{!*} j^* L \Bigr ).$
\end{nota}

\noindent \textit{Remark}: we have
$L/\Fil^0_{U,!}(L) \simeq i_* \lexp {p+} i^* L$ and 
$\lexp p j_{!*} j^* L \htarrow_+ \Fil^0_{U,!}(L)/\Fil^{-1}_{U,!}(L)$, which gives, cf.
lemma 1.3.13 of \cite{boyer-torsion}, a commutative triangle
$$\xymatrix{
\lexp p j_{!*} j^* L \ar@{^{(}->>}[rr]_+ \ar@{^{(}->>}[drr]_+ 
& & \Fil^0_{U,!}(L)/\Fil^{-1}_{U,!}(L) \ar@{^{(}->>}[d]_+ \\
& & \lexp {p+} j_{!*} j^* L.
}$$

Consider now $X$ equipped with a stratification 
$$X=X^{\geq 1} \supset X^{\geq 2} \supset \cdots \supset X^{\geq d,}$$
and let $L \in \FC(X,\overline \Zm_l)$.
For $1 \leq h <d$, let denote
$X^{1 \leq h}:=X^{\geq 1}-X^{\geq h+1}$ and $j^{1 \leq h}:X^{1†\leq h} \hookrightarrow
X^{\geq 1}$. We then define
$$\Fil^r_{!}(L):=\im_\FC \Bigl ( \lexp {p+} j^{1 \leq r}_! j^{1 \leq r,*}L \longrightarrow L\Bigr ),$$
which gives a filtration
$$0=\Fil^{0}_{!}(L) \subset \Fil^1_{!}(L) \subset \Fil^1_{!}(L) \cdots \subset \Fil^{d-1}_{!}(L) 
\subset \Fil^d_{!}(L)=L.$$

With these notations, the graduate $\gr^h_!(\Psi_{\varrho})$ verify
$$j^{=h,*} \gr^h_!(\Psi_{\varrho}) \simeq 
\left \{ \begin{array}{ll} 0 & \hbox{si } g \nmid h \\
\LC_{\overline \Zm_l}(\varrho[t]_D) & \hbox{pour } h=tg. \end{array} \right.$$
In particular for $h=d$, the perverse sheaf $\gr^d_!(\Psi_\varrho)$ is a quotient of $\Psi_\varrho$
and concentrated of the supersingular locus. One of the main result of \cite{boyer-local-ihara}
can be stated as follow.

\begin{prop} (cf. \cite{boyer-local-ihara} proposition 4.2.4)
There exists a filtration
$$0=\Fil^{-1-u}_\varrho \subset \Fil^{-u}_\varrho \subset \Fil^{1-u}_\varrho \subset \cdots 
\Fil^0_\varrho \subset \Fil^{1}_\varrho= \gr^d_!(\Psi_\varrho)$$
verifying the following properties.
\begin{itemize}
\item The graduates $\gr^{-k}_\varrho=\Fil^{-k}_\varrho/\Fil^{-k-1}_\varrho$ are free of $\varrho$-type
$k$ in the following sense
$$\gr^{-k}_\varrho \otimes_{\overline \Zm_l} \overline \Qm_l \simeq \bigoplus_{\pi_v \in \scusp_{k}(\varrho)}
HT(\pi_v,\st_{s_k(\varrho)}(\pi_v)) \otimes \Lm_{g_k(\varrho)}(\pi_v)(\frac{1-s_k(\varrho)}{2}),$$
where for $i \geq 0$ recall $g_k(\varrho)=g_{-1}m(\varrho)l^k$ and $g_k(\varrho)s_k(\varrho)=d$.

\item For any geometric supersingular point $i_z=\{ z \} \hookrightarrow X^{=d}_{\IC,\bar s_v}$, then
$$\ind_{(D_{v,d}^\times)^0 \varpi_v^\Zm}^{D_{v,d}^\times} \lexp p h^0 i_z^* \gr^{-k}_\varrho
\bigoplus_{\pi_v \in \scusp_k(\varrho)} \Gamma_G(\pi_v) \otimes \Gamma_{DW}(\pi_v)$$
where $\Gamma_G(\pi_v)$ (resp. $\Gamma_{DW}(\pi_v)$) is a stable lattice of 
$\st_{s_k(\varrho)}(\pi_v)$ (resp. of $\pi_v[s_k(\varrho)]_D \otimes 
\Lm_{g_k(\varrho)}(\pi_v)(\frac{1-s_k(\varrho)}{2})$.

\item Let $M_d(F_v)$ be the mirabolic subgroup of $GL_d(F_v)$ defined by imposing the first column
to be $(1,0,\cdots,0)$. Then if $\tau$ is an irreducible sub-$M_d(F_v)$-representation of
$\Gamma_G(\pi_v) \otimes_{\overline \Zm_l} \overline \Fm_l$, then $\tau$ is isomorphic to 
the non degenerate representation $\tau_{nd}$, unique up to isomorphism. As a consequence any
irreducible $GL_d(F_v)$-representation of $\Gamma_G(\pi_v) \otimes_{\overline \Zm_l} \overline \Fm_l$
is isomorphic to $\rho_u$.
\end{itemize}
\end{prop}

\rem About $\Gamma_{DW}(\pi_v)$, note that for $\pi_v \in \scusp_{-1}(\varrho)$ as the modulo $l$
reduction of $\pi_v[s_k(\varrho)]_D \otimes \Lm_{g_k(\varrho)}(\pi_v)(\frac{1-s_k(\varrho)}{2})$
is irreducible, up to isomorphism, there exists only one stable lattice which is then a tensor product
$\Gamma_D(\pi_v) \otimes \Gamma_W(\pi_v)$.

In \cite{boyer-torsion} proposition 2.3.3, we explained how, using $\Fil^{-1}(U,!)$, to construct 
an exhaustive filtration of $\Psi_\varrho$
\begin{multline*}
0=\Fill^{-2^{d-1}}_{!}(\Psi_\varrho) \subset \Fill^{-2^{d-1}+1}_{!}(\Psi_\varrho) 
\subset \cdots \subset \Fill^0_{!}(\Psi_\varrho) \subset \cdots \\ \subset 
\Fill^{2^{d-1}-1}_{!}(\Psi_\varrho)=\Psi_\varrho,
\end{multline*}
such that each of the graduate is over $\overline \Qm_l$, a direct sum of Harris-Taylor perverse
sheaves of the same Newton stratum. It's then possible to refine this filtration to obtain a new one
$$\Fill_{!,\varrho}^0 \subset \Fill_{!,\varrho}^1 \subset \cdots  \subset \Fill_{!,\varrho}^r$$
where each graduate is a Harris-Taylor perverse $\overline \Zm_l$-sheaf. Note the lattices constructed
in this way, may depend of the construction but concerning the quotient of the previous proposition,
the statement remains true. Precisely we can manage so that $\grr^r_{!,\varrho}:=\Fill^r_{!,\varrho}/
\Fill^{r-1}_{!,\varrho}$ verify, with the notation of the previous proposition
$$\ind_{(D_{v,d}^\times)^0 \varpi_v^\Zm}^{D_{v,d}^\times} \lexp p h^0 i_z^* \grr^r_{!,\varrho} \simeq
\Gamma_G(\pi_v) \otimes \Gamma_D(\pi_v) \otimes \Gamma_W(\pi_v),$$
where $\pi_v$ is any fixed irreducible cuspidal representation in $\scusp_{-1}(\varrho)$.

\subsection{Main result}
\label{para-torsion-classes1}

Start from an irreducible automorphic cuspidal representation $\Pi$ of $G(\Am)$ verifying the following
properties:
\begin{itemize}
\item it is $\xi$-cohomological with non trivial invariant under some fixed $I \in \IC$;

\item its degeneracy depth is equal to $s>1$;

\item its local component at $v$ is isomorphic to $\speh_s(\pi_v)$ with $\pi_v \in \scusp_{-1}(\varrho)$
and where $d=g_u(\varrho)$ for some $u \geq 0$.
\end{itemize}

\rem For $\pi_v$ the trivial characte, the hypothesis $d=g_u(\varrho)$ for $u=0$ is equivalent to ask
that the order of $q \in \Fm_l$, which is the cardinal of the residue field of $F_v$, is equal to $d$.
Another way to formulate this condition, is to say that the $L$-function of the trivial character modulo $l$,
has a pole at $s=1$.

Denote $\mathfrak m$ the maximal ideal of $\Tm_I$ associated to $\Pi$.
Remember that $\mathfrak m$ encodes the multiset of Satake's parameters of $\Pi$ outside $I$.

\begin{prop}
Under the previous hypothesis, the torsion of $H^1(X_{I,\bar \eta_v},\LC_\xi[d-1])_{\mathfrak m}$
is non trivial.
\end{prop}

\begin{proof}
Consider the spectral sequence 
$$E_1^{p,q}=H^{p+q}(X_{I,\bar s_v},\grr^{-p}_{!,\varrho})_{\mathfrak m} \Rightarrow
H^{p+q}(X_{I,\bar \eta_v},V_\xi)_{\mathfrak m}$$
associated to the filtration $\Fill^{\bullet}_{!,\varrho}$ of $\Psi_\varrho$.
First note that over $\Qm_l$:
\begin{itemize}
\item $E_1^{-r,r} \otimes_{\overline \Zm_l} \overline \Qm_l$ has a direct factor isomorphic to 
$(\Pi^{\oo,v})^I \otimes \st_s(\pi_v) \otimes \Lm_g(\pi_v) (\frac{1-s}{2})$;

\item $d_1^{-r,r} \otimes_{\overline \Zm_l} \overline \Qm_l$ induces a injection from the previous
direct factor into a direct factor of $E_1^{-r+1,r}$ which, as a representation of $GL_d(F_v)$
is parabolically induced from $P_{(s-1)g,d}(F_v)$ to $GL_d(F_v)$.
\end{itemize}
From the last remark of the previous section,  
$\ind_{(D_{v,d}^\times)^0 \varpi_v^\Zm}^{D_{v,d}^\times} \lexp p h^0 i_z^* E_1^{-r,r}$
as a $GL_d(F_v)$-representation, has a subspace isomorphic to
 $\Gamma_G(\pi_v)$
where $\Gamma_G(\pi_v)$ is a stable lattice of $\st_s(\pi_v)$
such that $\rho_u$ is the only irreducible sub-representation of 
$\Gamma_G(\pi_v) \otimes_{\overline \Zm_l} \overline \Qm_l$. Moreover we know that
$\rho_u$ can't be a subspace of parabolically induced representation.
From these facts we conclude that the torsion of $E_2^{-r+1,r}$ is non trivial and more precisely
that $\rho_u$ is a subquotient of $E_2^{-r+1,r}[l]$.

Now there are two alternatives. Either $\rho_u$ as a subquotient of $E_2^{-r+1,r}[l]$ remains a
subquotient of $E_\oo^1[l]$ and we are done. Suppose by absurdity it's not the case. 
First about the free quotient $E_{k,free}^{p,q}$ of the $E_k^{p,q}$, 
we know from \cite{boyer-compositio} that:
\begin{itemize}
\item if $\rho_u$ is a subquotient of $E_{1,free}^{p,q} \otimes_{\overline \Zm_l} \overline \Fm_l$
with $p+q \neq 0$, then $\grr_{!,\varrho)}^{-p}$ is isomorphic to some $\PC(\pi_v,s_i(\varrho)-1)$
with $\pi_v \in \scusp_i(\varrho)$ and then $p+q=\pm 1$;

\item for $k \geq 2$ and $p+q \neq 0$, then $\rho_u$ is never a subquotient of
$E_{k,free}^{p,q} \otimes_{\overline \Zm_l} \overline \Fm_l$.
\end{itemize}
Then there must exist $(p,q)$ and a torsion class in $(E_{1,tor}^{p,q})_{\mathfrak m}$ with
$p+q=2$ such that $\rho_u$ is a subquotient of its $l$-torsion which contradicts lemma \ref{lem-torsion-rel}.
\end{proof}

Thanks to the main result of \cite{boyer-mrl}, associated to this torsion cohomology class is a
 tempered irreducible automorphic representation $\Pi'$ of $G(\Am)$ which is
$\xi$-cohomological of level $I^l I'_l$ and weakly $\mathfrak m$-congruent to $\Pi$ in the sense 
it shares the same multiset of Satake's parameters than $\Pi$ outside $I$. In particular for $s=2$,
as in Ribet's proof of Herbrand theorem, we should obtain a non trivial element in the Selmer group
of the adjoint representation of the Galois $\overline \Fm_l$-representation associated to $\mathfrak m$.

\bibliographystyle{plain}
\bibliography{bib-ok}

\end{document}